\documentclass[12pt,reqno]{amsart}

\usepackage{color}
\usepackage{amsmath, amsthm, amssymb}
\usepackage{amsfonts}

\usepackage[ansinew]{inputenc}
\usepackage[dvips]{epsfig}
\usepackage{graphicx}
\usepackage[english]{babel}
\usepackage{hyperref}
\theoremstyle{plain}
\newtheorem{thm}{Theorem}[section]

\newtheorem{lem}[thm]{Lemma}
\newtheorem{prop}[thm]{Proposition}

\theoremstyle{definition}

\theoremstyle{remark}

\numberwithin{equation}{section}

\newcommand{\average}{{\mathchoice {\kern1ex\vcenter{\hrule height.4pt
width 6pt depth0pt} \kern-9.7pt} {\kern1ex\vcenter{\hrule
height.4pt width 4.3pt depth0pt} \kern-7pt} {} {} }}

\def\R{\mathbb{R}}

\begin{document}

\title[Boundary regularity, Pohozaev identities, nonexistence results]{Boundary regularity, Pohozaev identities\\ and nonexistence results}

\author{Xavier Ros-Oton}
\address{The University of Texas at Austin, Department of Mathematics, 2515 Speedway, Austin, TX 78751, USA}
\email{ros.oton@math.utexas.edu}


\thanks{The author was supported by NSF grant DMS-1565186 (US) and by MINECO grant MTM2014-52402-C3-1-P (Spain)}

\maketitle

\begin{abstract}
In this expository paper we survey some recent results on Dirichlet problems of the form $Lu=f(x,u)$ in $\Omega$, $u\equiv0$ in $\R^n\backslash\Omega$.
We first discuss in detail the boundary regularity of solutions, stating the main known results of Grubb and of the author and Serra.
We also give a simplified proof of one of such results, focusing on the main ideas and on the blow-up techniques that we developed in \cite{RS-K,RS-stable}.
After this, we present the Pohozaev identities established in \cite{RS-Poh,RSV,Grubb-Poh} and give a sketch of their proofs, which use strongly the fine boundary regularity results discussed previously.
Finally, we show how these Pohozaev identities can be used to deduce nonexistence of solutions or unique continuation properties.

The operators $L$ under consideration are integro-differential operator of order~$2s$, $s\in(0,1)$, the model case being the fractional Laplacian $L=(-\Delta)^s$.
\end{abstract}

\section{Introduction}

This expository paper is concerned with the study of solutions to
\begin{equation}\label{pb}
\left\{ \begin{array}{rcll}
L u &=&f(u)&\textrm{in }\Omega \\
u&=&0&\textrm{in }\R^n\backslash\Omega,
\end{array}\right.\end{equation}
where $\Omega\subset\R^n$ is a bounded domain, and $L$ is an elliptic integro-differential operator of the form
\begin{equation}\label{L}\begin{split}
Lu(x)=\textrm{P.V.}&\int_{\R^n}\bigl(u(x)-u(x+y)\bigr)K(y)dy,\\
K\geq0,\qquad K(y)=K(-y),&\qquad\textrm{and}\qquad \int_{\R^n}\min\bigl\{|y|^2,1\bigr\}K(y)dy<\infty.
\end{split}\end{equation}
Such operators appear in the study of stochastic process with jumps: L\'evy processes.
In the context of integro-differential equations, L\'evy processes play the same role that Brownian motion plays in the theory of second order PDEs.
In particular, the study of such processes leads naturally to problems posed in bounded domains like~\eqref{pb}.

Solutions to \eqref{pb} are critical points of the nonlocal energy functional
\[\mathcal E(u)=\frac12\int_{\R^n}\int_{\R^n}\bigl(u(x)-u(x+y)\bigr)^2K(y)dy\,dx-\int_\Omega F(u)dx\]
among functions $u\equiv0$ in $\R^n\setminus\Omega$.
Here, $F'=f$.

Here, we will work with operators $L$ of order $2s$, with $s\in(0,1)$.
In the simplest case we will have $K(y)=c_{n,s}|y|^{-n-2s}$, which corresponds to $L=(-\Delta)^s$, the fractional Laplacian.
More generally, a typical assumption would be
\[0<\frac{\lambda}{|y|^{n+2s}}\leq K(y)\leq \frac{\Lambda}{|y|^{n+2s}}.\]
Under such assumption, operators \eqref{L} can be seen as uniformly elliptic operators of order $2s$, for which Harnack inequality and other regularity properties are well understood; see for example \cite{R-Survey}.

For the Laplace operator, \eqref{pb} becomes
\begin{equation}\label{Laplacian}
\left\{ \begin{array}{rcll}
-\Delta u &=&f(u)&\textrm{in }\Omega \\
u&=&0&\textrm{on }\partial\Omega.
\end{array}\right.
\end{equation}
A model case for \eqref{Laplacian} is the power-type nonlinearity $f(u)=|u|^{p-1}u$, with $p>1$.
In this case, it is well known that the mountain pass theorem yields the existence of (nonzero) solutions for $p<\frac{n+2}{n-2}$, while for powers $p\geq \frac{n+2}{n-2}$ the only bounded solution in star-shaped domains is $u\equiv0$.
In other words, one has existence of solutions in the subcritical regime, and non-existence of solutions in star-shaped domains in the critical or supercritical regimes.

An important tool in the study of solutions to \eqref{Laplacian} is the \emph{Pohozaev identity} \cite{P}.
This celebrated result states that any bounded solution to this problem satisfies the identity
\begin{equation}\label{Poh-classic-1}
\int_\Omega\bigl\{2n\,F(u)-(n-2)u\,f(u)\bigr\}dx=\int_{\partial\Omega}\left(\frac{\partial u}{\partial \nu}\right)^2(x\cdot\nu)d\sigma(x),
\end{equation}
where
\[F(u)=\int_0^uf(t)dt.\]
When $f(u)=|u|^{p-1}u$ then the identity becomes
\[\left(\frac{2n}{p+1}-(n-2)\right)\int_\Omega|u|^{p+1}dx=\int_{\partial\Omega}\left(\frac{\partial u}{\partial \nu}\right)^2(x\cdot\nu)d\sigma(x).\]
When $p\geq\frac{n+2}{n-2}$, the left hand side of this identity is negative or zero, while the right hand side is strictly positive for nonzero solutions in star-shaped domains.
Thus, the nonexistence of solutions follows.

The proof of the identity \eqref{Poh-classic-1} is based on the following integration-by-parts type formula
\begin{equation}\label{Poh-classic-2}
2\int_\Omega (x\cdot\nabla u)\Delta u\,dx=(2-n)\int_\Omega u\,\Delta u\,dx+\int_{\partial\Omega}\left(\frac{\partial u}{\partial \nu}\right)^2(x\cdot\nu)d\sigma(x),
\end{equation}
which holds for any $C^2$ function with $u=0$ on $\partial\Omega$.
This identity is an easy consequence of the \emph{divergence theorem}.
Indeed, using that
\[\Delta (x\cdot\nabla u)=x\cdot\nabla \Delta u+2\Delta u\]
and that
\[x\cdot \nabla u=(x\cdot\nu)\frac{\partial u}{\partial \nu}\qquad \textrm{on} \quad\partial\Omega,\]
then integrating by parts (three times) we find
\[\begin{split}\int_\Omega (x\cdot\nabla u)\Delta u\,dx&=-\int_\Omega \nabla (x\cdot\nabla u)\cdot\nabla u\,dx+\int_{\partial\Omega}(x\cdot\nabla u)\frac{\partial u}{\partial \nu}\,d\sigma\\
& = \int_\Omega \Delta(x\cdot\nabla u)\,u\,dx+\int_{\partial\Omega}(x\cdot\nabla u)\frac{\partial u}{\partial \nu}\,d\sigma \\
& = \int_\Omega \{x\cdot\nabla \Delta u+2\Delta u\}u\,dx+\int_{\partial\Omega}(x\cdot\nu)\left(\frac{\partial u}{\partial \nu}\right)^2d\sigma\\
&= \int_\Omega \{-\textrm{div}(xu)\Delta u+2u\Delta u\}dx+\int_{\partial\Omega}(x\cdot\nu)\left(\frac{\partial u}{\partial \nu}\right)^2d\sigma \\
& = \int_\Omega \{-(x\cdot\nabla u)\Delta u-(n-2)u\Delta u\}dx+\int_{\partial\Omega}(x\cdot\nu)\left(\frac{\partial u}{\partial \nu}\right)^2d\sigma,
\end{split}\]
and hence \eqref{Poh-classic-2} follows.

Identities of Pohozaev-type like \eqref{Poh-classic-1} and \eqref{Poh-classic-2} have been used widely in the analysis of elliptic PDEs: they yield to monotonicity formulas, unique continuation properties, radial symmetry of solutions, and uniqueness results.
Moreover, they are also used in other contexts such as hyperbolic equations, harmonic maps, control theory, and geometry.

The aim of this paper is to show what are the \emph{nonlocal} analogues of these identities, explain the main ideas appearing in their proofs, and give some immediate consequences concerning the nonexistence of solutions.
Furthermore, we will also discuss a very related issue: the \emph{boundary regularity} of solutions.

\vspace{3mm}

$\bullet$ {\bf A simple case.}
In order to have a first hint on what should be the analogue of \eqref{Poh-classic-2} for integro-differential operators \eqref{L}, let us look at the simplest case $L=(-\Delta)^s$, and let us assume that $u\in C^\infty_c(\Omega)$.
In this case, a standard computation shows that
\[(-\Delta)^s(x\cdot\nabla u)=x\cdot\nabla(-\Delta)^su+2s\,(-\Delta)^su.\]
This is a pointwise equality that holds at every point $x\in\R^n$.
This, combined with the global integration by parts identity in all of $\R^n$
\begin{equation}\label{global}
\int_{\R^n}u\,(-\Delta)^sv\,dx=\int_{\R^n}(-\Delta)^su\,v\,dx,
\end{equation}
leads to
\begin{equation}\label{toy}
\qquad\qquad 2\int_\Omega(x\cdot \nabla u)(-\Delta)^s u\,dx=(2s-n)\int_\Omega u(-\Delta)^s u\,dx,\qquad \textrm{for}\quad u\in C^\infty_c(\Omega).
\end{equation}
Indeed, taking $v=x\cdot \nabla u$ one finds
\[\begin{split}
\int_{\R^n}(x\cdot\nabla u)(-\Delta)^su\,dx&=\int_{\R^n}u\,(-\Delta)^s(x\cdot \nabla u)\,dx\\
&=\int_{\R^n} u\left\{x\cdot\nabla(-\Delta)^su+2s\,(-\Delta)^su\right\}dx\\
&=\int_{\R^n} \left\{-\textrm{div}(xu)(-\Delta)^su+2s\,u(-\Delta)^su\right\}dx\\
&=\int_{\R^n} \left\{-(x\cdot\nabla u)(-\Delta)^su+(2s-n)\,u(-\Delta)^su\right\}dx,
\end{split}\]
and thus \eqref{toy} follows.

This identity has no boundary term (recall that we assumed that $u$ and all its derivatives are zero on $\partial\Omega$), but it is a first approximation towards a nonlocal version of \eqref{Poh-classic-2}.
The only term that is missing is the boundary term.

As we showed above, when $s=1$ and $u\in C^2_0(\overline\Omega)$, the use of the divergence theorem in $\Omega$ (instead of the global identity \eqref{global}) leads to the Pohozaev-type identity \eqref{Poh-classic-2}, with the boundary term.
However, in case of nonlocal equations there is no divergence theorem in bounded domains, and this is why at first glance there is no clear candidate for a nonlocal analogue of the boundary term in \eqref{Poh-classic-2}.

In order to get such a Pohozaev-type identity for solutions to \eqref{pb}, with the right boundary term, we first need to answer the following:
\[\textrm{What is the \emph{boundary regularity} of solutions to \eqref{pb}?}\]
Once this is well understood, we will come back to the study of Pohozaev identities and we will present the nonlocal analogues of \eqref{Poh-classic-1}-\eqref{Poh-classic-2} established in \cite{RS-Poh,RSV}.

\vspace{2mm}

The paper is organized as follows:

In Section \ref{sec2} we discuss the boundary regularity of solutions to \eqref{pb}.
We will state the main known results, and give a sketch of the proofs and their main ingredients.
Then, in Section \ref{sec3} we present the Pohozaev identities of \cite{RS-Poh,RSV} and give some ideas of their proofs.
Finally, in Section \ref{sec4} we give some consequences of such Pohozaev identities.

\section{Boundary regularity}
\label{sec2}

The study of integro-differential equations started already in the fifties with the works of Getoor, Blumenthal, and Kac, among others \cite{BGR,G}.
Due to the relation with L\'evy processes, they studied Dirichlet problems
\begin{equation}\label{linear}
\left\{ \begin{array}{rcll}
Lu &=&g(x)&\textrm{in }\Omega \\
u&=&0&\textrm{in }\R^n\backslash\Omega,
\end{array}\right.\end{equation}
and proved some basic properties of solutions, estimates for the Green function, and the asymptotic distribution of eigenvalues.
Moreover, in the simplest case of the fractional Laplacian $(-\Delta)^s$, the following explicit solutions were found:
\[\quad u_0(x)=(x_+)^s\qquad\quad\textrm{solves}\qquad
\left\{\begin{array}{rcl}
(-\Delta)^su_0\hspace{-2.5mm}&=0 & \quad\textrm{in}\quad (0,\infty) \\
u_0\hspace{-2.5mm}&=0 & \quad \textrm{in}\quad (-\infty,0)
\end{array}\right.\quad\]
and
\begin{equation}\label{explicit}
u_1(x)=\kappa_{n,s}\,\bigl(1-|x|^2\bigr)^s_+\qquad\textrm{solves}\qquad
\left\{\begin{array}{rcl}
(-\Delta)^su_0\hspace{-2.5mm}&=1 & \quad\textrm{in}\quad B_1 \\
u_0\hspace{-2.5mm}&=0 & \quad \textrm{in}\quad \R^n\setminus B_1,
\end{array}\right.
\end{equation}
for certain constant $\kappa_{n,s}$; see \cite{BV-book}.

The interior regularity of solutions for $L=(-\Delta)^s$ is by now well understood.
Indeed, potential theory for this operator enjoys an explicit formulation in terms of the Riesz potential, and thus it is similar to that of the Laplacian; see the classical book of Landkov \cite{L}.

For more general linear operators \eqref{L}, the interior regularity theory has been developed in the last years, and it is now quite well understood for operators satisfying
\begin{equation}\label{L0}
0<\frac{\lambda}{|y|^{n+2s}}\leq K(y)\leq \frac{\Lambda}{|y|^{n+2s}};
\end{equation}
see for example the results of Bass \cite{Bass}, Serra \cite{Se}, and also the survey \cite{R-Survey} for regularity results in H\"older spaces.

Concerning the boundary regularity theory for the fractional Laplacian, fine estimates for the Green's function near $\partial\Omega$ were established by Kulczycki \cite{Potential1} and Chen-Song \cite{Potential2}; see also \cite{BGR}.
These results imply that, in $C^{1,1}$ domains, all solutions $u$ to \eqref{linear} are comparable to $d^s$, where $d(x)=\textrm{dist}(x,\R^n\setminus\Omega)$.
More precisely,
\begin{equation}\label{Cd^s}
|u|\leq Cd^s
\end{equation}
for some constant $C$, and from this bound one can deduce an estimate of the form
\[\|u\|_{C^s(\overline\Omega)}\leq C\|g\|_{L^\infty(\Omega)}\]
for \eqref{linear}.
Moreover, when $g>0$, then $u\geq c\,d^s$ for some $c>0$ ---recall the example \eqref{explicit}.
In particular, solutions $u$ are $C^s$ up to the boundary, and this is the optimal H\"older exponent for the regularity of $u$, in the sense that in general $u\notin C^{s+\epsilon}(\overline\Omega)$ for any $\epsilon>0$.

More generally, when the equation and the boundary data are given only in a subregion of $\R^n$, one has the following estimate, whose proof we sketch below.
Notice that the following estimate is for a general class of nonlocal operators $L$, which includes the fractional Laplacian.

\begin{prop}[\cite{RS-stable}]\label{Cs-estimate}
Let $\Omega\subset\R^n$ be any bounded $C^{1,1}$ domain with $0\in\partial\Omega$, and $L$ be any operator \eqref{L}-\eqref{L0}, with $K(y)$ homogeneous.
Let $u$ be any bounded solution to
\[\left\{ \begin{array}{rcll}
Lu &=&g&\textrm{in }\Omega\cap B_1 \\
u&=&0&\textrm{in }B_1\backslash\Omega,
\end{array}\right.\]
with $g\in L^\infty(\Omega\cap B_1)$.
Then,\footnote{Here, we denote $\|w\|_{L^1_s(\R^n)}:=\int_{\R^n}\frac{w(x)}{1+|x|^{n+2s}}\,dx.$}
\[\|u\|_{C^s(B_{1/2})}\leq C\bigl(\|g\|_{L^\infty(\Omega\cap B_1)}+\|u\|_{L^\infty(B_1)}+\|u\|_{L^1_s(\R^n)}\bigr),\]
with $C$ depending only on $n$, $s$, $\Omega$, and ellipticity constants.
\end{prop}

\begin{figure}
\begin{center}
\includegraphics[]{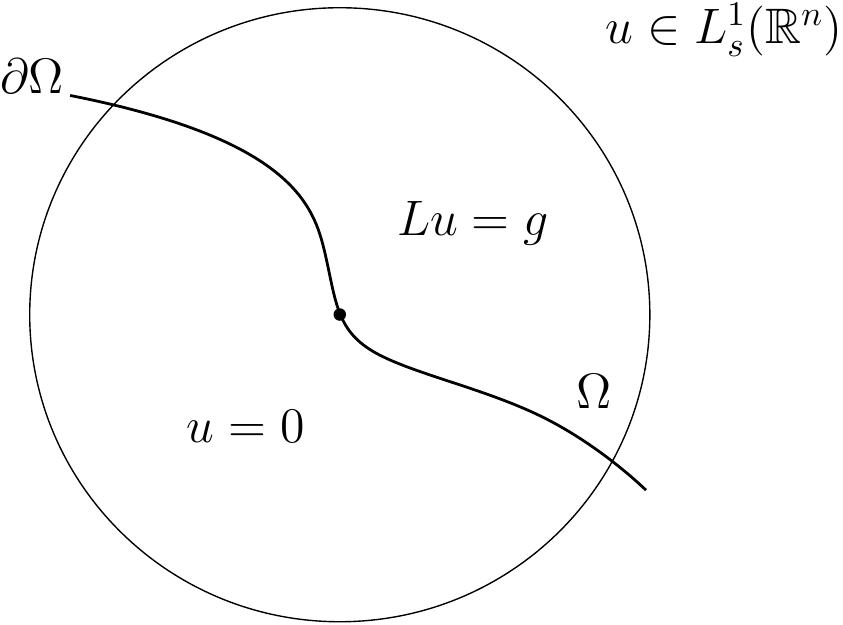}
\end{center}
\end{figure}

\begin{proof}
We give a short sketch of this proof.
For more details, see \cite{RS-Dir} or \cite{RS-stable}.

First of all, truncating $u$ and dividing it by a constant if necessary, we may assume that $\|g\|_{L^\infty(\Omega\cap B_1)}+\|u\|_{L^\infty(\R^n)}\leq 1$.
Second, by constructing a supersolution (using for example Lemma~\ref{lem-ds} below) one can show that
\begin{equation}\label{bound-ds}
|u|\leq Cd^s\qquad \textrm{in}\quad\Omega.
\end{equation}
Then, once we have this we need to show that
\begin{equation}\label{Cs}
|u(x)-u(y)|\leq C|x-y|^s\qquad \forall x,y\in \overline\Omega.
\end{equation}
We separate two cases, depending on whether $r=|x-y|$ is bigger or smaller than $\rho=\min\{d(x),d(y)\}$.

More precisely, if $2r\leq \rho$, then $y\in B_{\rho/2}(x)\subset B_\rho(x)\subset \Omega$ (we assume without loss of generality that $\rho=d(x)\leq d(y)$ here).
Therefore, one can use known interior estimates (rescaled) and \eqref{bound-ds} to get
\begin{equation}\label{int-est}
[u]_{C^s(B_{\rho/2}(x))}\leq C.
\end{equation}
Indeed, by \eqref{bound-ds} we have that the rescaled function $u_\rho(z):=u(x+\rho z)$ satisfies
\[\|u_\rho\|_{L^\infty(B_1)}\leq C\rho^s,\qquad \|u_\rho\|_{L^1_s(\R^n)}\leq C\rho^s,\qquad \textrm{and}\qquad \|Lu_\rho\|_{L^\infty(B_1)}\leq C\rho^{2s},\]
and therefore by interior estimates
\[\begin{split}
\rho^s[u]_{C^s(B_{\rho/2}(x))}& =[u_\rho]_{C^s(B_{1/2})}\leq C\bigl(\|u_\rho\|_{L^\infty(B_1)}+\|u_\rho\|_{L^1_s(\R^n)}+\|Lu_\rho\|_{L^\infty(B_1)}\bigr)\\
&\leq C(\rho^s+\rho^s+\rho^{2s})\leq C\rho^s.
\end{split}\]
In particular, it follows from \eqref{int-est} that $|u(x)-u(y)|\leq C|x-y|^s$.

On the other hand, in case $2r>\rho$ then we just use \eqref{bound-ds} to get
\[\begin{split}
|u(x)-u(y)|&\leq |u(x)|+|u(y)|\leq Cd^s(x)+Cd^s(y)\\
&\leq Cd^s(x)+C\bigl(d^s(x)+|x-y|^s\bigr)\leq C\rho^s+C(r+\rho)^s\\
&\leq Cr^s=C|x-y|^s.\end{split}\]
In any case, we get \eqref{Cs}, as desired.
\end{proof}

\subsection{Higher order boundary regularity estimates}

Unfortunately, in the study of Pohozaev identities the bound \eqref{Cd^s} is not enough, and finer regularity results are needed.
A more precise description of solutions near $\partial\Omega$ is needed.

For second order (local) equations, solutions to the Dirichlet problem are known to be $C^\infty(\overline\Omega)$ whenever $\Omega$ and the right hand side $g$ are $C^\infty$.
In case $g\in L^\infty(\Omega)$, then $u\in C^{2-\epsilon}(\overline\Omega)$ for all $\epsilon>0$.
This, in particular, yields a fine description of solutions $u$ near $\partial\Omega$: for any $z\in \partial\Omega$ one has
\[\bigl|u(x)-c_zd(x)\bigr|\leq C|x-z|^{2-\epsilon},\]
where $d(x)=\textrm{dist}(x,\Omega^c)$ and $c_z\in \R$.
This is an expansion of order $2-\epsilon$, which holds whenever $g\in L^\infty$ and $\Omega$ is $C^{1,1}$.
When $g$ and $\Omega$ are $C^\infty$, then one has analogue higher order expansions that essentially say that $u/d\in C^\infty(\overline\Omega)$.

\vspace{2mm}

The question for nonlocal operators was: are there any nonlocal analogues of such higher order boundary regularity estimates?

\vspace{2mm}

The first result in this direction was obtained by the author and Serra in \cite{RS-Dir} for the fractional Laplacian $L=(-\Delta)^s$; we showed that $u/d^s\in C^\alpha(\overline\Omega)$ for some small $\alpha>0$.
Such result was later improved and extended to more general operators by Grubb \cite{Grubb,Grubb2} and by the author and Serra \cite{RS-K,RS-stable}.
These results may be summarized as follows.

\begin{thm}[\cite{Grubb,Grubb2,RS-K,RS-stable}]\label{bdry}
Let $\Omega\subset\R^n$ be any bounded domain, and $L$ be any operator \eqref{L}-\eqref{L0}.
Assume in addition that $K(y)$ homogeneous, that is,
\begin{equation}\label{stable}
K(y)=\frac{a\left(y/|y|\right)}{|y|^{n+2s}}.
\end{equation}
Let $u$ be any bounded solution to \eqref{linear}, and\footnote{In fact, to avoid singularities inside $\Omega$, we define $d(x)$ as a positive function that coincides with $\textrm{dist}(x,\R^n\setminus\Omega)$ in a neighborhood of $\partial\Omega$ and is as regular as $\partial\Omega$ inside $\Omega$.} $d(x)={\rm dist}(x,\R^n\setminus\Omega)$.
Then,
\begin{itemize}
\item[(a)] If $\Omega$ is $C^{1,1}$, then
\[g\in L^\infty(\Omega)\quad \Longrightarrow\quad u/d^s\in C^{s-\epsilon}(\overline\Omega)\qquad\quad \textrm{for all}\ \epsilon>0,\]
\item[(b)] If $\Omega$ is $C^{2,\alpha}$ and $a\in C^{1,\alpha}(S^{n-1})$, then
\[\qquad g\in C^\alpha(\overline\Omega)\quad \Longrightarrow\quad u/d^s\in C^{\alpha+s}(\overline\Omega)\qquad \textrm{for small}\ \alpha>0,\]
whenever $\alpha+s$ is not an integer.
\item[(c)] If $\Omega$ is $C^\infty$ and $a\in C^\infty(S^{n-1})$, then
\[\quad g\in C^\alpha(\overline\Omega)\quad \Longrightarrow\quad u/d^s\in C^{\alpha+s}(\overline\Omega)\qquad \textrm{for all}\ \alpha>0,\]
whenever $\alpha+s\notin\mathbb Z$.
In particular, $u/d^s\in C^\infty(\overline\Omega)$ whenever $g\in C^\infty(\overline\Omega)$.
\end{itemize}
\end{thm}

It is important to remark that the above theorem is just a particular case of the results of \cite{Grubb,Grubb2} and \cite{RS-K,RS-stable}.
Indeed, part (a) was proved in \cite{RS-stable} for any $a\in L^1(S^{n-1})$ (without the assumption \eqref{L0}); (b) was established in \cite{RS-K} in the more general context of fully nonlinear equations; and (c) was established in \cite{Grubb,Grubb2} for all pseudodifferential operators satisfying the $s$-transmission property.
Furthermore, when $s+\alpha$ is an integer in (c), more information is given in \cite{Grubb2} in terms of H\"older-Zygmund spaces $C^k_*$.

When $g\in L^\infty(\Omega)$ and $\Omega$ is $C^{1,1}$, the above result yields a fine description of solutions $u$ near $\partial\Omega$: for any $z\in \partial\Omega$ one has
\[\bigl|u(x)-c_zd^s(x)\bigr|\leq C|x-z|^{2s-\epsilon},\]
where $d(x)=\textrm{dist}(x,\Omega^c)$.
This is an expansion of order $2s-\epsilon$, which is analogue to the one described above for the Laplacian.

In case of second order (local) equations, only the regularity of $g$ and $\partial\Omega$ play a role in the result.
In the nonlocal setting of operators of the form \eqref{L}-\eqref{L0}, a third ingredient comes into play: the regularity of the kernels $K(y)$ in the $y$-variable.
This is why in parts (b) and (c) of Theorem \ref{bdry} one has to assume some regularity of $K$.
This is a purely nonlocal feature, and cannot be avoided.
In fact, when the kernels are not regular then counterexamples to higher order regularity can be constructed, both to interior and boundary regularity; see \cite{Se,RS-stable}.
Essentially, when the kernels are not regular, one can expect to get regularity results up to order $2s$, but not higher.
We refer the reader to \cite{R-Survey}, where this is discussed in more detail.

\vspace{2mm}

Let us now sketch some ideas of the proof of Theorem \ref{bdry}.
We will focus on the simplest case and try to show the main ideas appearing in its proof.

\subsection{Sketch of the proof of Theorem \ref{bdry}(a)}

A first important ingredient in the proof of Theorem \ref{bdry}(a) is the following computation.

\begin{lem}\label{lem-ds}
Let $\Omega$ be any $C^{1,1}$ domain, $s\in (0,1)$, and $L$ be any operator of the form \eqref{L}-\eqref{L0} with $K(y)$ homogeneous, i.e., of the form \eqref{stable}.
Let $d(x)$ be any positive function that coincides with $\textrm{dist}(x,\R^n\setminus\Omega)$ in a neighborhood of $\partial\Omega$ and is $C^{1,1}$ inside $\Omega$.
Then,
\begin{equation}
|L(d^s)|\leq C_\Omega\qquad \textrm{in}\quad \Omega,
\end{equation}
where $C_\Omega$ depends only on $n$, $s$, $\Omega$, and ellipticity constants.
\end{lem}

\begin{proof}
Let $x_0\in \Omega$ and $\rho=d(x)$.
Notice that when $\rho\geq\rho_0>0$ then $d^s$ is $C^{1,1}$ in a neighborhood of $x_0$, and thus $L(d^s)(x_0)$ is bounded by a constant depending only on $\rho_0$.
Thus, we may assume that $\rho\in(0,\rho_0)$, for some small $\rho_0>0$.

Let us denote
\[\ell(x):=\bigl(d(x_0)+\nabla d^s(x_0)\cdot (x-x_0)\bigr)_+,\]
and notice that $\ell^s$ is a translated and rescaled version of the 1-D solution $(x_n)_+^s$.
Thus, we have
\[L(\ell^s)=0\qquad \textrm{in}\quad \{\ell>0\};\]
see \cite[Section 2]{RS-K}.
Moreover, notice that by construction of $\ell$ we have
\[d(x_0)=\ell(x_0)\qquad \textrm{and}\qquad \nabla d(x_0)=\nabla \ell(x_0).\]
Using this, it is not difficult to see that
\[\bigl|d(x_0+y)-\ell(x_0+y)\bigr|\leq C|y|^2,\]
and this yields
\[\bigl|d^s(x_0+y)-\ell^s(x_0+y)\bigr|\leq C|y|^2\bigl(d^{s-1}(x_0+y)+\ell^{s-1}(x_0+y)\bigr).\]

On the other hand, for $|y|>1$ we clearly have
\[\bigl|d^s(x_0+y)-\ell^s(x_0+y)\bigr|\leq C|y|^s\qquad \textrm{in}\quad \R^n\setminus B_1.\]

Using the last two inequalities, and recalling that $L(\ell^s)(x_0)=0$ and that $d(x_0)=\ell(x_0)$, we find
\[\begin{split}
\bigl|L(d^s)(x_0)\bigr|&=\bigl|L(d^s-\ell^s)(x_0)\bigr|=\left|\int_{\R^n}(d^s-\ell^s)(x_0+y)K(y)dy\right|\\
&\leq C\int_{B_1}|y|^2\bigl(d^{s-1}(x_0+y)+\ell^{s-1}(x_0+y)\bigr)\frac{dy}{|y|^{n+2s}} +C\int_{B_1^c}|y|^s\frac{dy}{|y|^{n+2s}}\\
&\leq C\int_{B_1}\bigl(d^{s-1}(x_0+y)+\ell^{s-1}(x_0+y)\bigr)\frac{dy}{|y|^{n+2s-2}}+C.
\end{split}\]
Such last integral can be bounded by a constant $C$ depending only on $s$ and $\Omega$, exactly as in \cite[Lemma 2.5]{RS-C1}, and thus it follows that
\[\bigl|L(d^s)(x_0)\bigr|\leq C,\]
as desired.
\end{proof}

Another important ingredient in the proof of Theorem \ref{bdry}(a) is the following classification result for solutions in a half-space.

\begin{prop}[\cite{RS-stable}]\label{prop-Liouville}
Let $s\in (0,1)$, and $L$ be any operator of the form \eqref{L}-\eqref{L0} with $K(y)$ homogeneous.
Let $u$ be any solution of
\begin{equation}\label{eq-I-flat}
\left\{ \begin{array}{rcll}
L v &=&0&\textrm{in }\{x\cdot e>0\} \\
v&=&0&\textrm{in }\{x\cdot e\leq0\}.
\end{array}\right.
\end{equation}
Assume that, for some $\beta<2s$, $u$ satisfies the growth control
\begin{equation}\label{gr-v}
|v(x)|\leq C\bigl(1+|x|^\beta\bigr)\qquad \textrm{in}\quad \R^n.
\end{equation}
Then,
\[v(x)=K(x\cdot e)_+^s\]
for some constant $K\in \R$.
\end{prop}

\begin{proof}
The idea is to differentiate $v$ in the directions that are orthogonal to $e$, to find that $v$ is a 1D function $v(x)=\bar v(x\cdot e)$.
Then, for 1D functions any operator $L$ with kernel \eqref{stable} is just a multiple of the 1D fractional Laplacian, and thus one only has to show the result in dimension 1.
Let us next explain the whole argument in more detail.

Given $R\geq1$ we define
\[v_R(x):=R^{-\beta}v(Rx).\]
It follows from the growth condition on $v$ that
\[\bigl| v_R(x)\bigr|\leq C\bigl(1+|x|^\beta\bigr)\qquad \textrm{in}\quad \R^n,\]
and moreover $v_R$ satisfies \eqref{eq-I-flat}, too.

Since $\beta<2s$, then $\|v_R\|_{L^1_s(\R^n)}\leq C$, and thus by Proposition \ref{Cs-estimate} we get
\[\|v_R\|_{C^s(B_{1/2})}\leq C,\]
with $C$ independent of $R$.
Therefore, using $[v]_{C^s(B_{R/2})}=R^{\beta-s}[v_R]_{C^s(B_{1/2})}$, we find
\begin{equation}\label{v-R}
\bigl[v\bigr]_{C^s(B_{R/2})}\leq CR^{\beta-s}\quad \textrm{for all}\quad R\geq1.
\end{equation}

Now, given $\tau\in S^{n-1}$ such that $\tau\cdot e=0$, and given $h>0$, consider
\[w(x):=\frac{v(x+h\tau)-v(x)}{h^s}.\]
By \eqref{v-R} we have that $w$ satisfies the growth condition
\[\|w\|_{L^\infty(B_R)}\leq CR^{\beta-s}\quad \textrm{for all}\quad R\geq1.\]
Moreover, since $\tau$ is a direction which is parallel to $\{x\cdot e=0\}$, then $w$ satisfies the same equation as $v$, namely $L w=0$ in $\{x\cdot e>0\}$, and $w=0$ in $\{x\cdot e\leq0\}$.
Thus, we can repeat the same argument above with $v$ replaced by $w$ (and $\beta$ replaced by $\beta-s$), to find
\[\bigl[w\bigr]_{C^s(B_{R/2})}\leq CR^{\beta-2s}\quad \textrm{for all}\quad R\geq1.\]

Since $\beta<2s$, letting $R\to\infty$ we find that
\[w\equiv0\quad \textrm{in}\quad \R^n.\]
This means that $v$ is a 1D function, $v(x)=\bar v(x\cdot e)$.
But then \eqref{eq-I-flat} yields that such function $\bar v:\R\to\R$ satisfies
\[\left\{ \begin{array}{rcll}
(-\Delta)^s\bar v &=&0&\textrm{in }(0,\infty) \\
\bar v&=&0&\textrm{in }(-\infty,0],
\end{array}\right.\]
with the same growth condition \eqref{gr-v}.
Using \cite[Lemma 5.2]{RS-K}, we find that $\bar v(x)=K(x_+)^s$, and thus
\[v(x)=K(x\cdot e)_+^s,\]
as desired.
\end{proof}

Using the previous results, let us now give the:

\vspace{3mm}

$\bullet$ {\bf Sketch of the proof of Theorem \ref{bdry}(a).}
In Proposition \ref{Cs} we saw how combining \eqref{bound-ds} with interior estimates (rescaled) one can show that $u\in C^s(\overline\Omega)$.
In other words, in order to prove the $C^s$ regularity up to the boundary, one only needs the bound $|u|\leq Cd^s$ and interior estimates.

Similarly, it turns out that in order to show that $u/d^s\in C^\gamma(\overline\Omega)$, $\gamma=s-\epsilon$, we just need an expansion of the form
\begin{equation}\label{expansion}
\left|u(x)-Q_z d^s(x)\right|\leq C|x-z|^{s+\gamma},\qquad z\in \partial\Omega,\qquad Q_z\in\R.
\end{equation}
Once this is done, one can combine \eqref{expansion} with interior estimates and get $u/d^s\in C^\gamma(\overline\Omega)$; see \cite[Proof of Theorem~1.2]{RS-stable} for more details.

Thus, we need to show \eqref{expansion}.

The proof of \eqref{expansion} is by contradiction, using a blow-up argument.
Indeed, assume that for some $z\in \partial\Omega$ the expansion \eqref{expansion} does not hold for any $Q\in\R$.
Then, we clearly have
\[\sup_{r>0}r^{-s-\gamma}\|u-Qd^s\|_{L^\infty(B_r(z))}=\infty\qquad \textrm{for all}\quad Q\in \R.\]
Then, one can show (see \cite[Lemma 5.3]{RS-stable}) that this yields
\[\sup_{r>0}r^{-s-\gamma}\|u-Q(r)d^s\|_{L^\infty(B_r(z))}=\infty,\qquad \textrm{with}\qquad Q(r)=\frac{\int_{B_r(z)}u\,d^s}{\int_{B_r(z)}d^{2s}}.\]
Notice that this choice of $Q(r)$ is the one which minimizes the $L^2$ distance between $u$ and $Qd^s$ in $B_r(z)$.

We define the monotone quantity
\[\theta(r):=\max_{r'\geq r}(r')^{-s-\gamma}\|u-Q(r')d^s\|_{L^\infty(B_{r'}(z))}.\]
Since $\theta(r)\rightarrow\infty$ as $r\rightarrow0$, then there exists a sequence $r_m\to0$ such that
\[(r_m)^{-s-\gamma}\|u-Q(r_m)d^s\|_{L^\infty(B_{r_m})}=\theta(r_m).\]

We now define the blow-up sequence
\[v_m(x):=\frac{u(z+r_mx)-Q(r_m)d^s(z+r_m x)}{(r_m)^{s+\gamma}\theta(r_m)}.\]
By definition of $Q(r_m)$ we have
\begin{equation}\label{contrad1}
\int_{B_1}v_m(x)\,d^s(z+r_m x)dx=0,
\end{equation}
and by definition of $r_m$ we have
\begin{equation}\label{contrad2}
\|v_m\|_{L^\infty(B_1)}=1
\end{equation}
Moreover, it can be shown that we have the growth control
\[\|v_m\|_{L^\infty(B_R)}\leq CR^{s+\gamma}\qquad\textrm{for all}\ R\geq1.\]
To prove this, one first shows that
\[|Q(Rr)-Q(r)|\leq C(rR)^{\gamma}\theta(r),\]
and then use the definitions of $v_m$ and $\theta$ to get
\[\begin{split}
\|v_m\|_{L^\infty(B_R)} &=\frac{1}{(r_m)^{s+\gamma}\theta(r_m)}\bigl\|u-Q(r_m)d^s\bigr\|_{L^\infty(B_{r_m R})}\\
&\leq \frac{1}{(r_m)^{s+\gamma}\theta(r_m)}\left\{\bigl\|u-Q(r_mR)d^s\bigr\|_{L^\infty(B_{r_m R})} +\bigl|Q(r_mR)-Q(r_m)\bigr|(r_mR)^s\right\}\\
&\leq \frac{1}{(r_m)^{s+\gamma}\theta(r_m)}\theta(r_mR)(r_mR)^{s+\gamma} +\frac{C}{(r_m)^{s+\gamma}\theta(r_m)}(r_mR)^\gamma\theta(r_m)(r_mR)^s\\
&\leq R^{s+\gamma}+CR^{s+\gamma}.
\end{split}\]
In the last inequality we used $\theta(r_mR)\leq \theta(r_m)$, which follows from the monotonicity of $\theta$ and the fact that $R\geq1$.

On the other hand, the functions $v_m$ satisfy
\[|Lv_m(x)|=\frac{(r_m)^{2s}}{(r_m)^{s+\gamma}\theta(r_m)}\bigl|Lu(z+r_m x)-L(d^s)(z+r_m x)\bigr|\qquad \textrm{in}\quad \Omega_m,\]
where the domain $\Omega_m=(r_m)^{-1}(\Omega-z)$ converges to a half-space $\{x\cdot e>0\}$ as $m\to\infty$.
Here $e\in S^{n-1}$ is the inward normal vector to $\partial\Omega$ at $z$.

Since $Lu$ and $L(d^s)$ are bounded, and $\gamma<s$, then it follows that
\[Lv_m\rightarrow 0\quad\textrm{uniformly in compact sets in}\ \{x\cdot e>0\}.\]
Moreover, $v_m\rightarrow0$ uniformly in compact sets in $\{x\cdot e<0\}$, since $u=0$ in $\Omega^c$.

Now, by $C^s$ regularity estimates up to the boundary and the Arzel\`a-Ascoli theorem the functions $v_m$ converge (up to a subsequence) to a function $v\in C(\R^n)$.
The convergence is uniform in compact sets of $\R^n$.
Therefore, passing to the limit the properties of $v_m$, we find
\begin{equation}\label{global1}
\|v\|_{L^\infty(B_R)}\leq CR^{s+\gamma}\qquad\textrm{for all}\ R\geq1,
\end{equation}
and
\begin{equation}\label{global2}
\left\{ \begin{array}{rcll}
Lv &=&0&\textrm{in }\{x\cdot e>0\} \\
v&=&0&\textrm{in }\{x\cdot e<0\}.
\end{array}\right.\end{equation}
Now, thanks to Proposition \ref{prop-Liouville}, we find
\begin{equation}\label{classif}
v(x)=K(x\cdot e)_+^s\qquad \textrm{for some}\quad K\in\R.
\end{equation}

Finally, passing to the limit \eqref{contrad1} ---using that $d^s(z+r_m x)/(r_m)^s\rightarrow (x\cdot e)_+ ^s$--- we find
\begin{equation}\label{contrad1bis}
\int_{B_1}v(x)\,(x\cdot e)_+^sdx=0,
\end{equation}
so that $K\equiv0$ and $v\equiv0$.
But then passing to the limit \eqref{contrad2} we get a contradiction, and hence \eqref{expansion} is proved.
\qed

\vspace{3mm}

It is important to remark that in \cite{RS-stable} we show \eqref{expansion} with a constant $C$ depending only on $n$, $s$, $\|g\|_{L^\infty}$, the $C^{1,1}$ norm of $\Omega$, and ellipticity constants.
To do that, the idea of the proof is exactly the same, but one needs to consider sequences of functions $u_m$, domains $\Omega_m$, points $z_m\in \partial\Omega_m$, and operators $L_m$.

\subsection{Comments, remarks, and open problems}
Let us next give some final comments and remarks about Theorem \ref{bdry}, as well as some related open problems.

\vspace{3mm}

$\bullet$ {\bf Singular kernels.}
Theorem \ref{bdry} (a) was proved in \cite{RS-stable} for operators $L$ with general homogeneous kernels of the form \eqref{stable} with $a\in L^1(S^{n-1})$, not necessarily satisfying \eqref{L0}.
In fact, $a$ could even be a singular measure.
In such setting, it turns out that Lemma \ref{lem-ds} is in general false, even in $C^\infty$ domains.
Because of this difficulty, the proof of Theorem \ref{bdry}(a) given in \cite{RS-stable} is in fact somewhat more involved than the one we sketched above.

\vspace{3mm}

$\bullet$ {\bf Counterexamples for non-homogeneous kernels.}
All the results above are for kernels $K$ satisfying \eqref{L0} and such that $K(y)$ is \emph{homogeneous}.
As said above, for the interior regularity theory one does not need the homogeneity assumption: the interior regularity estimates are the same for homogeneous or non-homogeneous kernels.
However, it turns out that something different happens in the boundary regularity theory.
Indeed, for operators with $x$-dependence
\[Lu(x)=\textrm{P.V.}\int_{\R^n}\bigl(u(x)-u(x+y)\bigr)K(x,y)dy,\]
\[0<\frac{\lambda}{|y|^{n+2s}}\leq K(x,y)\leq \frac{\Lambda}{|y|^{n+2s}}, \qquad K(x,y)=K(x,-y),\]
we constructed in \cite{RS-K} solutions to $Lu=0$ in $\Omega$, $u=0$ in $\R^n\setminus\Omega$, that are \emph{not} comparable to $d^s$.
More precisely, we showed that in dimension $n=1$ there are $\beta_1<s<\beta_2$ for which the functions $(x_+)^{\beta_i}$,  solve an equation of the form $Lu=0$ in $(0,\infty)$, $u=0$ in $(-\infty,0]$.
Thus, no fine boundary regularity like Theorem \ref{bdry} can be expected for non-homogeneous kernels; see \cite[Section 2]{RS-K} for more details.

\vspace{3mm}

$\bullet$ {\bf On the proof of Theorem \ref{bdry} (b).}
The proof of Theorem \ref{bdry}(b) in \cite{RS-K} has a similar structure as the one sketched above, in the sense that we show first $L(d^s)\in C^\alpha(\overline\Omega)$ and then prove an expansion of order $2s+\alpha$ similar to \eqref{expansion}.
However, there are extra difficulties coming from the fact that we would get exponent $2s+\alpha$ in \eqref{global1}, and thus the operator $L$ is not defined on functions that grow that much.
Thus, the blow-up procedure needs to be done with incremental quotients, and the global equation \eqref{global2} is replaced by \cite[Theorem 1.4]{RS-K}.

\vspace{3mm}

$\bullet$ {\bf On the proof of Theorem \ref{bdry} (c).}
Theorem \ref{bdry}(c) was proved in \cite{Grubb,Grubb2} by Fourier transform methods, completely different from the techniques presented above.
Namely, the results in \cite{Grubb,Grubb2} are for general pseudodifferential operators satisfying the so-called $s$-transmission property.
A key ingredient in those proofs is the existence of a factorization of the principal symbol, which leads to the boundary regularity properties for such operators.

\vspace{3mm}

$\bullet$ {\bf Open problem: Regularity in $C^{k,\alpha}$ domains.}
After the results of \cite{Grubb,Grubb2,RS-K,RS-stable}, a natural question remains open: what happens in $C^{k,\alpha}$ domains?

Our results in \cite{RS-K,RS-stable} give sharp regularity estimates in $C^{1,1}$ and $C^{2,\alpha}$ domains ---Theorem \ref{bdry} (a) and (b)---, while the results of Grubb \cite{Grubb,Grubb2} give higher order estimates in $C^\infty$ domains ---Theorem \ref{bdry} (c).
It is an open problem to establish sharp boundary regularity results in $C^{k,\alpha}$ domains, with $k\geq3$, for operators \eqref{L}-\eqref{L0} with homogeneous kernels.

For the fractional Laplacian, sharp estimates in $C^{k,\alpha}$ domains have been recently established in \cite{JN}, by using the extension problem for the fractional Laplacian.
For more general operators, this is only known for $k=1$ \cite{RS-C1} and $k=2$ \cite{RS-K}.

The development of sharp boundary regularity results in $C^{k,\alpha}$ domains for integro-differential operators $L$ would lead to the higher regularity of the free boundary in obstacle problems such operators; see \cite{DS}, \cite{JN}, \cite{CRS-obst}.

\vspace{3mm}

$\bullet$ {\bf Open problem: Parabolic equations.}
Part (a) of Theorem \ref{bdry} was recently extended to parabolic equations in \cite{FR}.
A natural open question is to understand the higher order boundary regularity of solutions for parabolic equations of the form
\[\partial_t u+Lu=f(t,x).\]
Are there analogous estimates to those in Theorem \ref{bdry} (b) and (c) in the parabolic setting?

This could lead to the higher regularity of the free boundary in parabolic obstacle problems for integro-differential operators; see \cite{CF,BFR2}.

\vspace{3mm}

$\bullet$ {\bf Open problem: Operators with different scaling properties.}
An interesting open problem concerning the boundary regularity of solutions is the following:
What happens with operators \eqref{L} with kernels having a different type of singularity near $y=0$\,?
For example, what happens with operators with kernels $K(y)\approx |y|^{-n}$ for $y\approx 0$\,?
This type of kernels appear when considering geometric stable processes; see \cite{SSV}.
The interior regularity theory has been developed by Kassmann-Mimica in \cite{KM} for very general classes of kernels, but much less is known about the boundary regularity; see \cite{BGR2} for some results in that direction.

\section{Pohozaev identities}
\label{sec3}

Once the boundary regularity is known, we can now come back to the Pohozaev identities.
We saw in the previous section that solutions $u$ to
\begin{equation}\label{pb2}
\left\{ \begin{array}{rcll}
L u &=&f(x,u)&\textrm{in }\Omega \\
u&=&0&\textrm{in }\R^n\backslash\Omega.
\end{array}\right.\end{equation}
are not $C^1$ up to the boundary, but the quotient $u/d^s$ is H\"older continuous up to the boundary.
In particular, for any $z\in \partial\Omega$ there exists the limit
\[\frac{u}{d^s}(z):=\lim_{\Omega \ni x\rightarrow z}\frac{u(x)}{d^s(x)}.\]
As we will see next, this function $u/d^s|_{\partial\Omega}$ plays the role of the normal derivative $\partial u/\partial\nu$ in the nonlocal analogues of \eqref{Poh-classic-2}-\eqref{Poh-classic-1}.

\begin{thm}[\cite{RS-Poh,RSV}]\label{thpoh}
Let $\Omega$ be any bounded $C^{1,1}$ domain, and $L$ be any operator of the form \eqref{L}, with
\[K(y)=\frac{a\left(y/|y|\right)}{|y|^{n+2s}}.\]
and $a\in L^\infty(S^{n-1})$.
Let $f$ be any locally Lipschitz function, $u$ be any bounded solution to \eqref{pb2}.
Then, the following identity holds
\begin{equation}\label{Poh1}
\hspace{-3mm}-2\int_\Omega(x\cdot\nabla u)Lu\ dx=(n-2s)\int_{\Omega}u\,Lu\, dx+\Gamma(1+s)^2\int_{\partial\Omega}\mathcal A(\nu)\left(\frac{u}{d^s}\right)^2(x\cdot\nu)d\sigma.
\end{equation}
Moreover, for all $e\in \R^n$, we have\footnote{In \eqref{Poh2}, we have corrected the sign on the boundary contribution, which
was incorrectly stated in \cite[Theorem 1.9]{RS-Poh}.}
\begin{equation}\label{Poh2}
-\int_\Omega \partial_eu\,Lu\,dx=\frac{\Gamma(1+s)^2}{2}\int_{\partial\Omega}\mathcal A(\nu)\left(\frac{u}{d^{s}}\right)^2(\nu\cdot e)\,d\sigma.
\end{equation}
Here
\begin{equation}\label{A}
\mathcal A(\nu)=c_s\int_{S^{n-1}}|\nu\cdot\theta|^{2s}a(\theta)d\theta,
\end{equation}
$a(\theta)$ is the function in \eqref{stable}, and $c_s$ is a constant that depends only on $s$.
For $L=(-\Delta)^s$, we have $\mathcal A(\nu)\equiv 1$.
\end{thm}

When the nonlinearity $f(x,u)$ does not depend on $x$, the previous theorem yields the following analogue of \eqref{Poh-classic-1}
\[\int_\Omega\bigl\{2n\,F(u)-(n-2s)u\,f(u)\bigr\}dx=\Gamma(1+s)^2\int_{\partial\Omega}\mathcal A(\nu)\left(\frac{u}{d^s}\right)^2(x\cdot\nu)d\sigma(x).\]
Before our work \cite{RS-Poh}, no Pohozaev identity for the fractional Laplacian was known (not even in dimension $n=1$).
Theorem \ref{thpoh} was first found and established for $L=(-\Delta)^s$ in \cite{RS-Poh}, and later the result was extended to more general operators in~\cite{RSV}.
A surprising feature of this result is that, even if the operators \eqref{L} are nonlocal, the identities \eqref{Poh1}-\eqref{Poh2} have completely local boundary terms.

Let us give now a sketch of the proof of the Pohozaev identity \eqref{Poh1}.
In order to focus on the main ideas, no technical details will be discussed.

\subsection{Sketch of the proof}

For simplicity, let us assume that $\Omega$ is $C^\infty$ and that $u/d^s\in C^\infty(\overline\Omega)$.

\vspace{2mm}

\noindent \underline{\emph{Step 1}}.
We first assume that $\Omega$ is strictly star-shaped; later we will deduce the general case from this one.
Translating $\Omega$ if necessary, we may assume it is strictly star-shaped with respect to the origin.

Let us define
\[u_\lambda(x) = u(\lambda x),\qquad\lambda>1,\]
and let us write the right hand side of \eqref{Poh1} as
\[ 2\int_\Omega(x\cdot \nabla u) Lu = 2\left.\frac{d}{d\lambda}\right|_{\lambda=1^+} \int_\Omega u_\lambda Lu.\]
This follows from the fact that $\left.\frac{d}{d\lambda}\right|_{\lambda=1^+} u_\lambda(x)=(x\cdot\nabla u)$ \ and the dominated convergence theorem.
Then, since $u_\lambda$ vanishes outside $\Omega$, we will have
\[\int_\Omega u_\lambda Lu=\int_{\R^n} u_\lambda Lu=\int_{\R^n}L^{\frac 12}u_\lambda L^{\frac 12} u,\]
and therefore
\begin{eqnarray*}
\int_\Omega u_\lambda Lu=\int_{\R^n}L^{\frac 12}u_\lambda L^{\frac 12} u
&=& \lambda^s\int_{\R^n}\left(L^{\frac 12}u\right)(\lambda x)L^{\frac 12}u(x)\,dx \\
&=& \lambda^s\int_{\R^n}w(\lambda x)w(x)\,dx\\
&=& \lambda^{\frac{2s-n}{2}}\int_{\R^n}w(\lambda^{\frac12}y)w(\lambda^{-\frac12}y)\,dy
\end{eqnarray*}
where  $w(x)= L^{\frac 12} u(x)$.

Now, since $2\left.\frac{d}{d\lambda}\right|_{\lambda=1^+}\lambda^{\frac{2s-n}{2}}=2s-n$,
the previous identities (and the change $\sqrt{\lambda}\mapsto\lambda$) yield
\[2\int_\Omega(x\cdot \nabla u) Lu=(2s-n)\int_{\R^n} w^2+\left.\frac{d}{d\lambda}\right|_{\lambda=1^+}\int_{\R^n}w_{{\lambda}}w_{1/{\lambda}}.\]
Moreover, since
\[\int_{\R^n} w^2=\int_{\R^n} L^{1/2}u\,L^{1/2}u=\int_{\R^n} u\,Lu=\int_{\Omega} u\,Lu,\]
then we have
\begin{equation}\label{step1}
-2\int_\Omega(x\cdot \nabla u) Lu=(n-2s)\int_{\Omega} u\,Lu+\mathcal I(w),
\end{equation}
where
\begin{equation}\label{operator}
\mathcal I(w)=-\left.\frac{d}{d\lambda}\right|_{\lambda=1^+}\int_{\R^n}w_{{\lambda}}w_{1/{\lambda}},
\end{equation}
$w_\lambda(x)=w(\lambda x)$, and $w(x)= L^{\frac 12} u(x)$.

At this point one should compare \eqref{Poh1} and \eqref{step1}.
In order to establish \eqref{Poh1}, we ``just'' need to show that $\mathcal I(w)$ is exactly the boundary term we want.

Let us take a closer look at the operator defined by \eqref{operator}.
The first thing one may observe by differentiating under the integral sign is that
\[\varphi\  \textrm{is ``nice enough''}\qquad \Longrightarrow\qquad \mathcal I(\varphi)=0.\]
In particular, one can also show that $\mathcal I(\varphi+h)=\mathcal I(\varphi)$ whenever $h$ is ``nice enough''.

The function $w=L^{1/2}u$ is smooth inside $\Omega$ and also in $\R^n\setminus\overline\Omega$, but it has a singularity along $\partial\Omega$.
In order to compute $\mathcal I(w)$, we have to study carefully the behavior of $w=L^{1/2}u$ near $\partial\Omega$, and try to compute $\mathcal I(w)$ by using \eqref{operator}.
The idea is that, since $u/d^s$ is smooth, then we will have
\begin{equation}\label{cosa1}
w=L^{1/2}u=L^{1/2}\left(d^s\frac{u}{d^s}\right)=L^{1/2}\bigl(d^s\bigr)\frac{u}{d^s}+\textrm{``nice terms''},
\end{equation}
and thus the behavior of $w$ near $\partial\Omega$ will be that of $L^{1/2}\bigl(d^s\bigr)\frac{u}{d^s}$.

Using the previous observation, and writing the integral in \eqref{operator} in the ``star-shaped coordinates'' \ $x=tz$, \ $z\in \partial\Omega$, \ $t\in(0,\infty)$, we find
\[\begin{split}
-\mathcal I(w)=\left.\frac{d}{d\lambda}\right|_{\lambda=1^+}\int_{\mathbb R^n}w_{{\lambda}}w_{1/{\lambda}}&=
\left.\frac{d}{d\lambda}\right|_{\lambda=1^+}\int_{\partial\Omega}(z\cdot \nu)d\sigma(z)\int_0^\infty t^{n-1}w(\lambda tz)w\left(\frac{tz}{\lambda}\right)dt\\
&=\int_{\partial\Omega}(z\cdot \nu)d\sigma(z)\left.\frac{d}{d\lambda}\right|_{\lambda=1^+}\int_0^\infty t^{n-1}w(\lambda tz)w\left(\frac{tz}{\lambda}\right)dt.
\end{split}\]
\begin{figure}
\begin{center}
\includegraphics[]{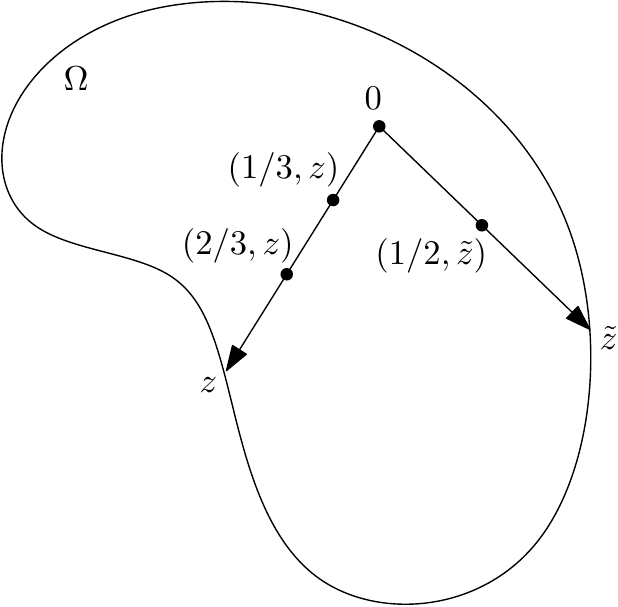}
\end{center}
\caption{\label{figura2}Star-shaped coordinates $x=tz$, with $z\in \partial\Omega$. }
\end{figure}

Now, a careful analysis of $L^{1/2}(d^s)$ leads to the formula
\begin{equation}\label{delicate0}
L^{1/2}\bigl(d^s\bigr)(tz)=\phi_s(t)\sqrt{\mathcal A(\nu(z))}+\textrm{``nice terms''},
\end{equation}
where $\phi_s(t)=c_1\{\log^-|t-1|+c_2\,\chi_{(0,1)}(t)\}$, and $c_1,c_2$ are explicit constants that depend only on $s$.
Here, $\chi_A$ denotes the characteristic function of the set $A$.

This, combined with \eqref{cosa1}, gives
\begin{equation}\label{delicate}
w(tz)=\phi_s(t)\sqrt{\mathcal A(\nu(z))}\,\frac{u}{d^s}(z)+\textrm{``nice terms''}.
\end{equation}

Using the previous two identities we find
\begin{equation}\label{last-step}\begin{split}
\mathcal I(w)&=
-\int_{\partial\Omega}(z\cdot \nu)d\sigma(z)\left.\frac{d}{d\lambda}\right|_{\lambda=1^+}\int_0^\infty t^{n-1}w(\lambda tz)w\left(\frac{tz}{\lambda}\right)dt\\
&=-\int_{\partial\Omega}(z\cdot \nu) d\sigma(z)\left.\frac{d}{d\lambda}\right|_{\lambda=1^+}\int_0^\infty t^{n-1}\phi_s(\lambda t)\phi_s\left(\frac{t}{\lambda}\right)\mathcal A(\nu(z))\left(\frac{u}{d^s}(z)\right)^2dt\\
&=\int_{\partial\Omega}\mathcal A(\nu)\left(\frac{u}{d^s}\right)^2 (z\cdot \nu)d\sigma(z)\,C(s),
\end{split}\end{equation}
where
\[C(s)=-\left.\frac{d}{d\lambda}\right|_{\lambda=1^+}\int_0^\infty t^{n-1}\phi_s(\lambda t)\phi_s\left(\frac{t}{\lambda}\right)dt\]
is a (positive) constant that can be computed explicitly.
Thus, \eqref{Poh1} follows from \eqref{step1} and~\eqref{last-step}.

\vspace{2mm}

\noindent \underline{\emph{Step 2}}.
Let now $\Omega$ be any $C^{1,1}$ domain.
In that case, the above proof does not work, since the assumption that $\Omega$ was star-shaped was very important in such proof.
Still, as shown next, once the identity \eqref{Poh1} is established for star-shaped domains, then the identity for general $C^{1,1}$ domains follows from an argument involving a partition of unity and the fact that every $C^{1,1}$ domain is locally star-shaped.

Let $B_i$ be a finite collection of small balls covering $\Omega$.
Then, we consider a family of functions $\psi_i\in C^\infty_c(B_i)$ such that $\sum_i \psi_i=1$, and we let $u_i=u\psi_i$.

We claim that for every $i,j$ we have the following bilinear identity
\begin{equation}\label{bilinear}\begin{split}
-\int_\Omega(x\cdot\nabla &u_i)Lu_j\, dx-\int_\Omega(x\cdot\nabla u_j)Lu_i\,dx=\frac{n-2s}{2}\int_\Omega u_iLu_j\, dx\,+\\
&+\frac{n-2s}{2}\int_{\Omega}u_jLu_i\, dx+\Gamma(1+s)^2\int_{\partial\Omega}\mathcal A(\nu)\frac{u_i}{\delta^{s}}\frac{u_j}{\delta^{s}}(x\cdot\nu)\, d\sigma.
\end{split}\end{equation}
To prove this, we separate two cases.
In case $\overline B_i\cap \overline B_j\neq\emptyset$ then it turns out that $u_i$ and $u_j$ satisfy the hypotheses of Step 1, and thus they satisfy the identity \eqref{Poh1} ---here we are using that the intersection of the $C^{1,1}$ domain $\Omega$ with a small ball is always star-shaped.
Then, applying \eqref{Poh1} to the functions $(u_i+u_j)$ and $(u_i-u_j)$ and subtracting such two identities, one gets \eqref{bilinear}.
On the other hand, in case $\overline B_i\cap \overline B_j=\emptyset$ then the identity \eqref{bilinear} is a simple computation similar to \eqref{toy}, since in this case we have $u_iu_j=0$ and thus there is no boundary term in \eqref{bilinear}.
Hence, we got \eqref{bilinear} for all $i,j$.
Therefore, summing over all $i$ and all $j$ and using that $\sum_i u_i=u$, \eqref{Poh1} follows.

\vspace{2mm}

\noindent \underline{\emph{Step 3}}.
Let us finally show the second identity \eqref{Poh2}.
For this, we just need to apply the identity that we already proved, \eqref{Poh1}, with a different origin $e\in\R^n$.
We get
\begin{equation}\label{Poh1-origin}
\begin{split}
-2\int_\Omega\bigl((x-e)\cdot\nabla u\bigr)Lu\ dx=&\,(n-2s)\int_{\Omega}u\,Lu\, dx\\
&+\Gamma(1+s)^2\int_{\partial\Omega}\mathcal A(\nu)\left(\frac{u}{d^s}\right)^2\bigl((x-e)\cdot\nu\bigr)d\sigma.
\end{split}\end{equation}
Subtracting \eqref{Poh1} and \eqref{Poh1-origin} we get \eqref{Poh2}, as desired.
\qed

\subsection{Comments and further results}

Let us next give some final remarks about Theorem \ref{thpoh}.

\vspace{3mm}

$\bullet$ {\bf On the proof of Theorem \ref{thpoh}.}
First, notice that the smoothness of $u/d^s$ and $\partial\Omega$ is hidden in \eqref{delicate}.
In fact, the proof of \eqref{delicate0}-\eqref{delicate} requires a very fine analysis, even if one assumes that both $u/d^s$ and $\partial\Omega$ are $C^\infty$.
Furthermore, even in this smooth case, the ``nice terms'' in \eqref{delicate} are not even $C^1$ near $\partial\Omega$, and a delicate result for $\mathcal I$ is needed in order to ensure that $\mathcal I(\textrm{``nice terms''})=0$; see Proposition~1.11 in \cite{RS-Poh}.

Second, note that the kernel of the operator $L^{1/2}$ has an explicit expression in case $L=(-\Delta)^s$, but not for general operators with kernels \eqref{stable}.
Because of this, the proofs of \eqref{delicate0} and \eqref{delicate} are simpler for $L=(-\Delta)^s$, and some new ideas are required to treat the general case, in which we obtain the extra factor $\sqrt{\mathcal A(\nu(z))}$.

\vspace{3mm}

$\bullet$ {\bf Extension to more general operators.}
After the results of \cite{RS-Poh,RSV}, a last question remained to be answered: what happens with more general operators \eqref{L}? For example, is there any Pohozaev identity for the class of operators $(-\Delta+m^2)^s$, with $m>0$? And for operators with $x$-dependence?

In a recent work \cite{Grubb-Poh}, G.\ Grubb obtained integration-by-parts formulas as in Theorem \ref{thpoh} for pseudodifferential operators $P$ of the form
\begin{equation}\label{psido}
Pu=\operatorname{Op}(p(x,\xi))u=\mathcal F^{-1}_{\xi \to x}(p(x,\xi )(\mathcal Fu)(\xi )),
\end{equation}
where $\mathcal F$ is the Fourier transform $(\mathcal F u)(\xi )=\int_{\R^n}e^{-ix\cdot \xi }u(x)\, dx$.
The symbol $p(x,\xi)$ has an asymptotic expansion $p(x,\xi )\sim\sum_{j\in{\Bbb N}_0}p_j(x,\xi )$ in homogeneous terms:
$p_j(x,t\xi)=t^{2s-j}p_j(x,\xi)$, and $p$ is {\it even} in the sense that $p_j(x,-\xi)=(-1)^j p_j(x,\xi)$ for all $j$.

When $a$ in \eqref{stable} is $C^\infty (S^{n-1})$, then the operators \eqref{L}-\eqref{stable} are pseudodifferential operators of the form \eqref{psido}.
For these operators \eqref{L}-\eqref{stable}, the lower-order terms $p_j$ ($j\ge 1$) vanish and $p_0$ is real and $x$-independent.
Here $p_0(\xi)=\mathcal F_{y\to \xi }K(y)$, and $\mathcal A(\nu)=p_0(\nu)$.
The fractional Laplacian $(-\Delta)^s$ corresponds to $a\equiv1$ in \eqref{L}-\eqref{stable}, and to $p(x,\xi)=|\xi|^{2s}$ in \eqref{psido}.

In case of operators \eqref{psido} with no $x$-dependence and with real symbols $p(\xi)$, the analogue of \eqref{Poh1} proved in \cite{Grubb-Poh} is the following identity
\[\begin{split}
-2\int_\Omega(x\cdot\nabla u)Pu\ dx=&\,\Gamma(1+s)^2\int_{\partial\Omega}p_0(\nu)\left(\frac{u}{d^s}\right)^2(x\cdot\nu)d\sigma\,+\\
&\qquad\qquad\qquad\qquad+n\int_{\Omega}u\,Pu\, dx-\int_{\Omega}u\,\operatorname{Op}(\xi\cdot\nabla p(\xi))u\,dx,
\end{split}\]
where $p_0(\nu)$ is the principal symbol of $P$ at $\nu$.
Note that when the symbol $p(\xi)$ is homogeneous of degree $2s$ (hence equals $p_0(\xi)$), then $\xi\cdot\nabla p(\xi)=2s\,p(\xi)$, and thus we recover the identity \eqref{Poh1}.

The previous identity can be applied to operators $(-\Delta+m^2)^s$.
Furthermore, the results in \cite{Grubb-Poh} allow $x$-dependent operators $P$, which result in extra integrals over~$\Omega$.
The methods in \cite{Grubb-Poh} are complex and quite different from the ones we use in \cite{RS-Poh,RSV}.
The domain $\Omega$ is assumed $C^\infty$ in~\cite{Grubb-Poh}.

\section{Nonexistence results and other consequences}
\label{sec4}

As in the case of the Laplacian $\Delta$, the Pohozaev identity \eqref{Poh1} gives as an immediate consequence the following nonexistence result for operators \eqref{L}-\eqref{stable}: If $f(u)=|u|^{p-1}u$ in \eqref{pb}, then
\begin{itemize}
\item If $\Omega$ is star-shaped and $p=\frac{n+2s}{n-2s}$, the only nonnegative weak solution is $u\equiv0$.
\item If $\Omega$ is star-shaped and $p>\frac{n+2s}{n-2s}$, the only bounded weak solution is $u\equiv0$.
\end{itemize}
This nonexistence result was first established by Fall and Weth in \cite{FW} for $L=(-\Delta)^s$.
They used the extension property of the fractional Laplacian, combined with the method of moving spheres.

On the other hand, the existence of solutions for subcritical powers $1<p<\frac{n+2s}{n-2s}$ was proved by Servadei-Valdinoci \cite{SV} for the class of operators \eqref{L}-\eqref{L0}.
Moreover, for the critical power $p=\frac{n+2s}{n-2s}$, the existence of solutions in an annular-type domains was obtained in \cite{SSS}.

\vspace{3mm}

The methods introduced in \cite{RS-Poh} to prove the Pohozaev identity \eqref{Poh1} were used in \cite{RS-nonex} to show nonexistence results for much more general operators $L$, including for example the following.

\begin{prop}[\cite{RS-nonex}]\label{corlevy}
Let $L$ be any operator of the form
\begin{equation}\label{levy}
Lu(x)=-\sum_{i,j}a_{ij}\partial_{ij}u+{\rm PV}\int_{\R^n}\bigl(u(x)-u(x+y)\bigr)K(y)dy,
\end{equation}
where $(a_{ij})$ is a positive definite symmetric matrix and $K$ satisfies the conditions in \eqref{L}.
Assume in addition that
\begin{equation}\label{condicio2levy}
K(y)|y|^{n+2}\ \textrm{is nondecreasing along rays from the origin.}
\end{equation}
and that
\[|\nabla K(y)|\leq C\,\frac{K(y)}{|y|}\quad \textrm{for all}\ y\neq0,\]
Let $\Omega$ be any bounded star-shaped domain, and $u$ be any bounded solution of \eqref{pb} with $f(u)=|u|^{p-1}u$.
If $p\geq\frac{n+2}{n-2}$, then $u\equiv0$.
\end{prop}

Similar nonexistence results were obtained in \cite{RS-nonex} for other types of nonlocal equations, including: kernels without homogeneity (such as sums of fractional Laplacians of different orders), nonlinear operators (such as fractional $p$-Laplacians), and operators of higher order ($s>1$).

\vspace{3mm}

Finally, let us give another immediate consequence of the Pohozaev identity \eqref{Poh1}.

\begin{prop}[\cite{RSV}]\label{cor-uniquecont}
Let $L$ be any operator of the form \eqref{L}-\eqref{L0}-\eqref{stable}, $\Omega$ be any bounded $C^{1,1}$ domain, and $\phi$ be any bounded solution to
\[\left\{ \begin{array}{rcll}
L \phi &=&\lambda\phi&\textrm{in }\Omega \\
\phi&=&0&\textrm{in }\R^n\backslash\Omega,
\end{array}\right.\]
for some real $\lambda$.
Then, $\phi/d^s$ is H\"older continuous up to the boundary, and the following unique continuation principle holds:
\[\left.\frac{\phi}{d^s}\right|_{\partial\Omega}\equiv0\quad \textrm{on}\quad \partial\Omega\qquad \Longrightarrow\qquad \phi\equiv0\quad \textrm{in}\quad \Omega.\]
\end{prop}

The same unique continuation property holds for any \emph{subcritical} nonlinearity $f(x,u)$; see Corollary~1.4 in \cite{RSV}.


\begin{thebibliography}{00}

\bibitem{Bass} R. Bass, \emph{Regularity results for stable-like operators}, J. Funct. Anal. 257 (2009), 2693–2722.

\bibitem{BFR2} B. Barrios, A. Figalli, X. Ros-Oton, \emph{Free boundary regularity in the parabolic obstacle problem for the fractional Laplacian}, preprint arXiv (May 2016).

\bibitem{BGR} R. M. Blumenthal, R. K. Getoor, D. B. Ray, \emph{On the distribution of first hits for the symmetric stable processes}, Trans. Amer. Math. Soc. 99 (1961), 540-554.

\bibitem{BGR} K. Bogdan, T. Grzywny, M. Ryznar, \emph{Heat kernel estimates for the fractional Laplacian with Dirichlet conditions}, Ann. of Prob. 38 (2010), 1901-1923.

\bibitem{BGR2} K. Bogdan, T. Grzywny, M. Ryznar, \emph{Barriers, exit time and survival probability for unimodal L\'evy processes}, Probab. Theory Relat. Fields 162 (2015), 155-198.

\bibitem{BV-book} C. Bucur, E. Valdinoci, \emph{Nonlocal Diffusion and Applications}, Lecture Notes of the Unione Matematica Italiana, Vol. 20, 2016.

\bibitem{CF} L. Caffarelli, A. Figalli, \emph{Regularity of solutions to the parabolic fractional obstacle problem}, J. Reine Angew. Math., 680 (2013), 191-233.

\bibitem{CRS-obst} L. Caffarelli, X. Ros-Oton, J. Serra, \emph{Obstacle problems for integro-differential operators: regularity of solutions and free boundaries}, Invent. Math., to appear.

\bibitem{Potential2} Z.-Q. Chen, R. Song, \emph{Estimates on Green functions and Poisson kernels for symmetric stable processes}, Math. Ann. 312 (1998), 465-501.

\bibitem{DS} D. De Silva, O. Savin, \emph{A note on higher regularity boundary Harnack inequality}, Disc. Cont. Dyn. Syst. 35 (2015), 6155-6163.

\bibitem{FW} M. M. Fall, T. Weth, \emph{Nonexistence results for a class of fractional elliptic boundary value problems}, J. Funct. Anal. 263 (2012), 2205-2227.

\bibitem{FR} X. Fernandez-Real, X. Ros-Oton. \emph{Regularity theory for general stable operators: parabolic equations}, preprint arXiv (Nov. 2015).

\bibitem{G} R. K. Getoor, \emph{First passage times for symmetric stable processes in space}, Trans. Amer. Math. Soc. 101 (1961), 75–90.

\bibitem{Grubb} G. Grubb, \emph{Fractional Laplacians on domains, a development of H\"ormander's theory of $\mu$-transmission pseudodifferential operators}, Adv. Math. 268 (2015), 478-528.

\bibitem{Grubb2} G. Grubb, \emph{Local and nonlocal boundary conditions for $\mu$-transmission and fractional elliptic pseudodifferential operators}, Anal. PDE 7 (2014), 1649-1682.

\bibitem{Grubb-Poh} G. Grubb, \emph{Integration by parts and Pohozaev identities for space-dependent fractional-order operators}, J. Differential Equations 261 (2016), 1835-1879.

\bibitem{JN} Y. Jhaveri, R. Neumayer, \emph{Higher regularity of the free boundary in the obstacle problem for the fractional Laplacian}, preprint arXiv (Jun. 2016).

\bibitem{KM} M. Kassmann, A. Mimica, \emph{Intrinsic scaling properties for nonlocal operators}, J. Eur. Math. Soc. (JEMS), to appear.

\bibitem{Potential1} T. Kulczycki, \emph{Properties of Green function of symmetric stable processes}, Probab. Math. Statist. 17 (1997), 339–364.

\bibitem{L} N. S. Landkof, \emph{Foundations of Modern Potential Theory}, Springer, New York, 1972.

\bibitem{P} S. I. Pohozaev, \emph{On the eigenfunctions of the equation $\Delta u + \lambda f(u) = 0$}, Dokl. Akad. Nauk SSSR 165 (1965), 1408-1411.

\bibitem{R-Survey} X. Ros-Oton, \emph{Nonlocal elliptic equations in bounded domains: a survey}, Publ. Mat. 60 (2016), 3-26.

\bibitem{RS-Dir} X. Ros-Oton, J. Serra, \emph{The Dirichlet problem for the fractional Laplacian: regularity up to the boundary}, J. Math. Pures Appl. 101 (2014), 275-302.

\bibitem{RS-Poh} X. Ros-Oton, J. Serra, \emph{The Pohozaev identity for the fractional Laplacian}, Arch. Rat. Mech. Anal 213 (2014), 587-628.

\bibitem{RS-nonex} X. Ros-Oton, J. Serra. \emph{Nonexistence results for nonlocal equations with critical and supercritical nonlinearities}, Comm. Partial Differential Equations 40 (2015), 115-133.

\bibitem{RS-K} X. Ros-Oton, J. Serra, \emph{Boundary regularity for fully nonlinear integro-differential equations}, Duke Math. J. 165 (2016), 2079-2154.

\bibitem{RS-stable} X. Ros-Oton, J. Serra, \emph{Regularity theory for general stable operators}, J. Differential Equations 260 (2016), 8675-8715.

\bibitem{RS-C1} X. Ros-Oton, J. Serra, \emph{Boundary regularity estimates for nonlocal elliptic equations in $C^1$ and $C^{1,\alpha}$ domains}, preprint arXiv (Dec. 2015).

\bibitem{RSV} X. Ros-Oton, J. Serra, E. Valdinoci, \emph{Pohozaev identities for anisotropic integro-differential operators}, preprint arXiv (Feb. 2015).


\bibitem{SSS} S. Secchi, N. Shioji, M. Squassina, \emph{Coron problem for fractional equations}, Differential Integral Equations 28 (2015), 103-118.

\bibitem{Se} J. Serra, \emph{$C^{\sigma+\alpha}$ regularity for concave nonlocal fully nonlinear elliptic equations with rough kernels}, Calc. Var. Partial Differential Equations 54 (2015), 3571-3601.

\bibitem{SV} R. Servadei, E. Valdinoci, \emph{Mountain pass solutions for non-local elliptic operators}, J. Math. Anal. Appl. 389 (2012), 887-898.

\bibitem{SSV} H. Sikic, R. Song, Z. Vondracek, \emph{Potential theory of geometric stable processes}, Probab. Theory Relat. Fields 135 (2006), 547-575.

\end{thebibliography}
\end{document}